\documentclass{amsart}
\usepackage{amssymb,amsmath,latexsym,stmaryrd}

\usepackage{graphicx}

\newtheorem{theorem}{Theorem}[section]   
\newtheorem{corollary}[theorem]{Corollary}     
\newtheorem{lemma}[theorem]{Lemma}         

\theoremstyle{definition}
\newtheorem{definition}[theorem]{Definition}   

\theoremstyle{remark}
\newtheorem{example}[theorem]{Example}        

\newcommand{\lb}{\llbracket}
\newcommand{\rb}{\rrbracket}
\newcommand{\f}{\bar{f}}
\newcommand{\x}{\bar{x}}
\newcommand{\y}{\bar{y}}

\usepackage{color}

\begin{document}

\title[Steady states of continuous and discrete models]{On the relationship of steady states of continuous and discrete models}
\author[Veliz-Cuba]{Alan Veliz-Cuba}
\author[Arthur]{Joseph Arthur}
\author[Hochstetler]{Laura Hochstetler}
\author[Klomps]{Victoria Klomps}
\author[Korpi]{Erikka Korpi}


\begin{abstract}
In this paper we provide theoretical results that relate steady states of continuous and discrete models arising from biology.
\end{abstract}
\maketitle

\section{Introduction}

Mathematical modeling has been used successfully to understand the behaviour of many biochemical networks, \cite{p53sode,p53ode,vondassow,santillan2008,p53lm,Velizlacop,mendoza_Tcell,albert}. Mathematical models can be used as a framework to identify the key parts of a biological system, to perform \textit{in silico} experiments and to understand the dynamical behaviour.

There are mainly two mathematical frameworks used in modeling biological systems. The continuous framework, usually constructed by coupling differential equations, \cite{p53sode,p53ode,vondassow,santillan2008}; and the discrete framework, usually constructed by coupling multistate functions (e.g. Boolean functions), \cite{p53lm,Velizlacop,mendoza_Tcell,albert}.

For many biological systems that have been modeled using both frameworks, it has been shown that both models have similar dynamical properties, \cite{p53sode,p53lm,vondassow,albert,santillan2008,Velizlacop}; this raises the question of when and why this happens. Furthermore, it has been hypothesized that the dynamics of biochemical networks are constrained by the topological structure of the wiring diagram, \cite{thomasbook}. This in turn suggests that mathematical models describing the same biochemical phenomena should have similar behaviour, even if they come from different frameworks. By providing answers to these questions we can gain substantial biological insight and understand how biology works at the system level.

This problem has been studied by several authors, \cite{kauffman-glass73,BtoC,mendoza,plinear}. Roughly speaking, the idea is that if a continuous model is ``sigmoidal enough'' and has the same qualitative features of a discrete model, then both models should have similar dynamics. In this paper we provide simple proofs about the relationship between steady states of continuous and discrete models satisfying these properties. Our results generalize previous results in \cite{BtoC,kauffman-glass73,mendoza,plinear}. Furthermore, our proofs also show why the relationship between continuous and discrete models is likely to occur even when continuous models are not very sigmoidal.

\section{Preliminaries}
\label{sec:pre}

For a biochemical network with $N$ species, a mathematical model is typically of the form $x(t)=(x_1(t),\ldots,x_N(t))$; where $x_i(t)$ denotes the concentration of species $i$ at time $t$. Given an initial condition $x(0)=x_0$, it is of interest to study the behaviour of $x(t)$ as $t$ increases. There is usually no closed form of $x_i(t)$ as a function of $t$ (and $x_0$), but rather equations that describe how $x_i$'s depend on each other.

A system of ordinary differential equations (ODE) that models a biochemical network is typically of the form $x_i'(t)=f_i(x(t))-d_i x_i(t)$, where $d_i$ is the elimination rate constant, corresponding to natural decay, and $f_i$ is a function that determines how the rate of change of $x_i$ at time $t$ depends on the value of the other species. We can represent this ODE model as $x'=f(x)-Dx$, where $f:\mathbb{R}^N\rightarrow\mathbb{R}^N$ and  $D$ is a diagonal matrix.

A multistate network or multivalued network (MN) that models a biochemical network is typically of the form  $x_i(t+1)=f_i(x(t))$ (notice that time is discrete); where $x_i$ can take values in the finite set $\lb 0,m_i\rb=\{0,1,\ldots,m_i\}$ and $f_i:\mathcal{S}=\lb 0,m_1\rb \times \lb 0,m_2\rb \times \ldots\times \lb 0,m_N\rb \rightarrow \lb 0,m_i\rb$ determines how the future value of species $i$ depends on the current value of the others. We can represent this MN model as $x(t+1)=f(x(t))$, where $f:\mathcal{S}\rightarrow\mathcal{S}$. A multistate network is also called a finite dynamical system. Examples of multistate networks include Boolean networks, logical models and polynomial dynamical systems \cite{Velizpoly}.

For any modeling framework, the wiring diagram or dependency graph is defined as the graph with $N$ vertices, $\{1,\ldots,N\}$, and an edge from $j$ to $i$ if $f_i$ depends on $x_j$. The edge is assigned a positive (negative) sign if $f_i$ is increasing (decreasing) with respect to $x_j$.

From now on we will use the notation $\bar{x}$ and $\bar{f}$ to refer to continuous entities; $x$ and $f$ will denote discrete entities. Also, $N$ will denote the dimension of functions.  $\|\ \|$ will denote the Euclidean norm for vectors and the Frobenius norm for matrices. We denote closed intervals by $[a,b]$, open intervals by $]a,b[$ and half-open intervals by $[a,b[$ and $]a,b]$.

Our goal is to relate the steady states of an ODE and a MN; that is, fixed points of $\f$ and $f$. We assume (rescaling if necessary) that the ODE has the form $\bar{x}'=D(\bar{f}(x)-\bar{x})$ where $D$ is a positive diagonal matrix and $\bar{f}:[0,1]^N\rightarrow [0,1]^N$. A MN is of the form $f:\mathcal{S}=\prod_{j=1}^N \lb 0,m_j\rb\rightarrow \mathcal{S}$.

\begin{example}\label{eg:toy}
Our toy example consists of a 2-dimensional ODE, $\x'=\f(\x)-\x$, and a 2-dimensional MN, $x(t+1)=f(x(t))$; where $\f:[0,1]^2\rightarrow [0,1]^2$ and $f:\lb 0,2\rb ^2\rightarrow \lb 0,2\rb^2$ are given by:

$$\f_1(\x)=.8\frac{\x_1^{n_1}}{.3^{n_1}+\x_1^{n_1}}+
.6\frac{\x_2^{n_2}}{.7^{n_2}+\x_2^{n_2}}\frac{.3^{n_3}}{.3^{n_3}+\x_1^{n_3}}$$
$$\f_2(\x)=.9\frac{\x_2^{n_4}}{.4^{n_4}+\x_2^{n_4}} +
.5\frac{\x_1^{n_5}}{.6^{n_5}+\x_1^{n_5}}\frac{.4^{n_6}}{.4^{n_6}+\x_2^{n_6}}$$ where $n=(n_1,\ldots,n_6)$ is a parameter. Also, $f$ is given as a truth table:

\begin{center}
$
\begin{array}{cc|cc}
x_1 & x_2 & f_1(x) & f_2(x)\\
\hline
0 & 0 & 0 & 0\\
0 & 1 & 0 & 2\\
0 & 2 & 1 & 2\\
1 & 0 & 2 & 0\\
1 & 1 & 2 & 2\\
1 & 2 & 2 & 2\\
2 & 0 & 2 & 1\\
2 & 1 & 2 & 2\\
2 & 2 & 2 & 2
\end{array} 
$
\end{center}

\end{example}

Notice that for our toy example $D$ is the identity matrix. We will refer to the discrete values 0, 1 and 2, as low, medium and high, respectively.

Figure \ref{fig:f12} shows plots for $\bar{f}_1$ and $\bar{f}_2$. We can see that $\bar{f}$ and $f$ have the same qualitative features. For example, when the value of both inputs is ``low'', the value of both outputs is low; when the value of both inputs is high, value of both outputs is high. For a better comparison, heat maps for all functions are shown in Figure \ref{fig:H1} and Figure \ref{fig:H2}. 

\begin{figure}[here]
\centerline{
\framebox{\includegraphics[height=4cm,width=8cm]{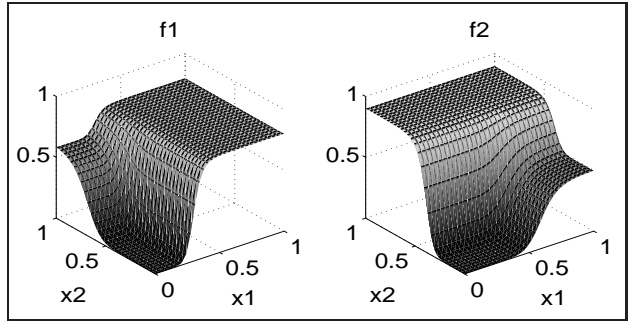}}
}
\caption{Graphs of functions $\f_1$ (left) and $\f_2$ (right)  from Example \ref{eg:toy}. All exponents $n_i$'s have been set equal to 10 for an easier comparison with the MN.}
\label{fig:f12}
\end{figure}

\begin{figure}[here]
\centerline{
\framebox{\includegraphics[height=3cm,width=3.5cm]{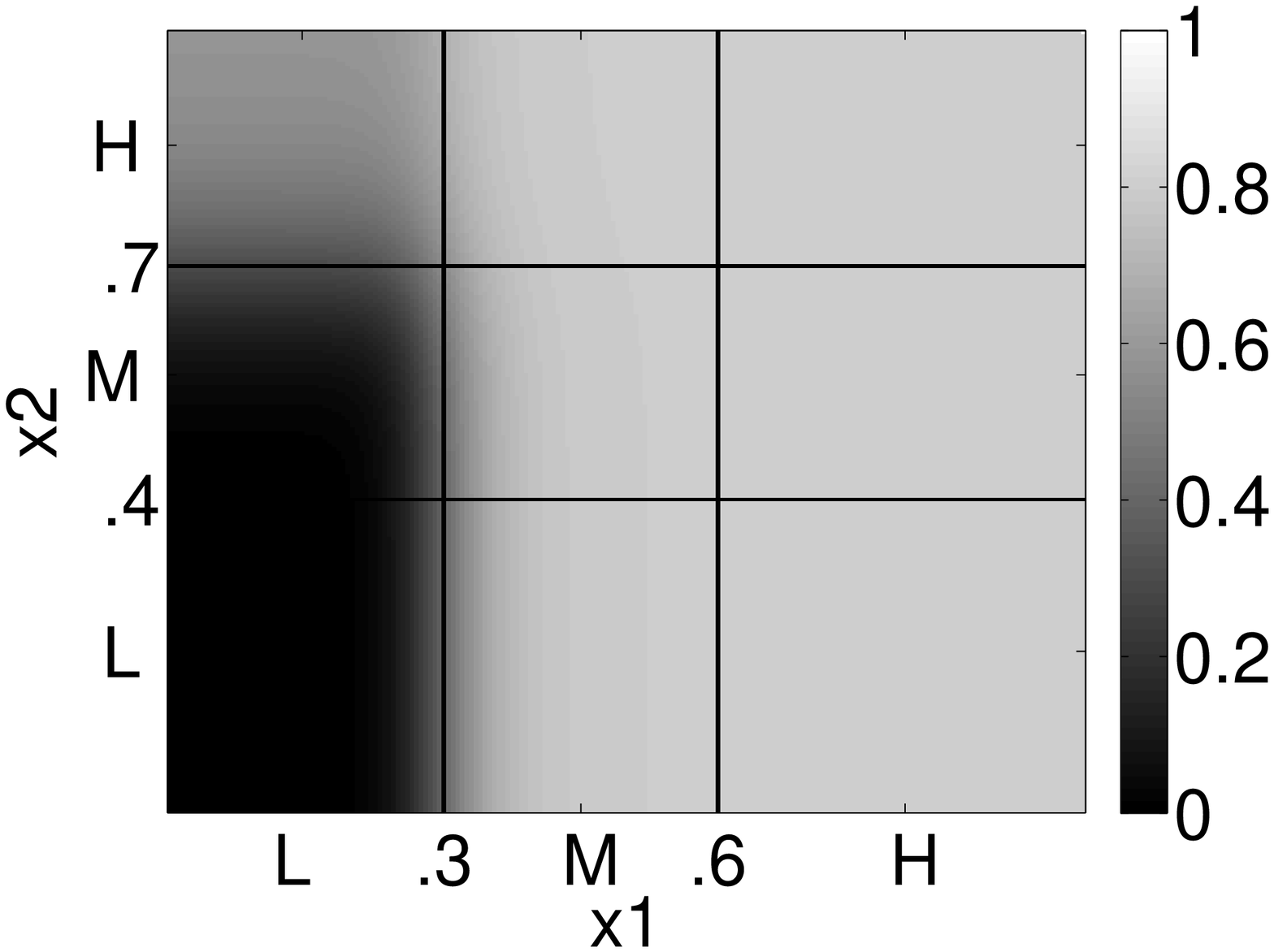} \ \ \includegraphics[height=3cm,width=3.5cm]{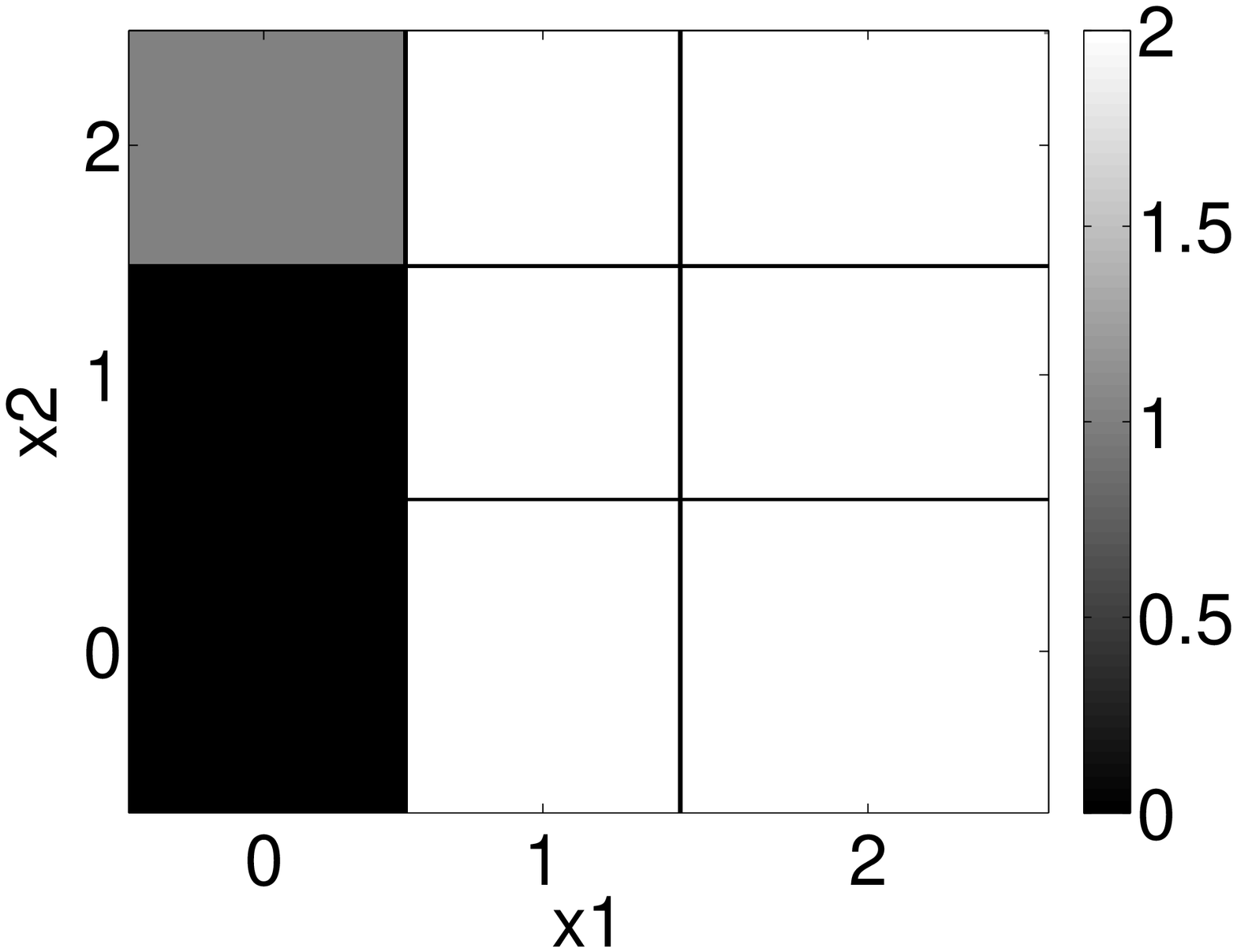}}
}
\caption{Heat maps for function $\f_1$ (left) and $f_1$ (right) from Example \ref{eg:toy}. L, M and H denote low, medium and high, respectively.}
\label{fig:H1}
\end{figure}

\begin{figure}[here]
\centerline{
\framebox{\includegraphics[height=3.3cm,width=3.5cm]{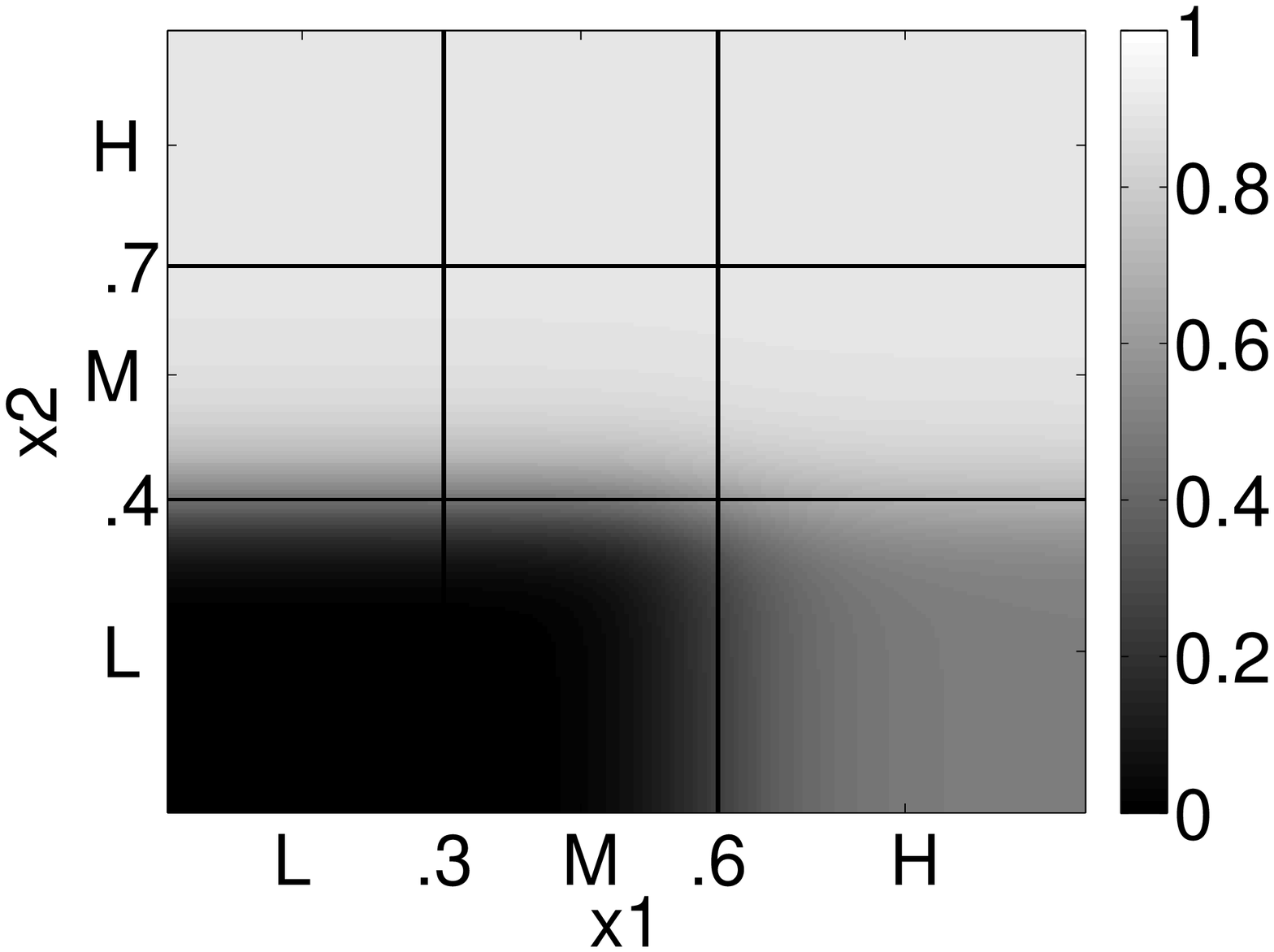} \ \ \includegraphics[height=3.3cm,width=3.5cm]{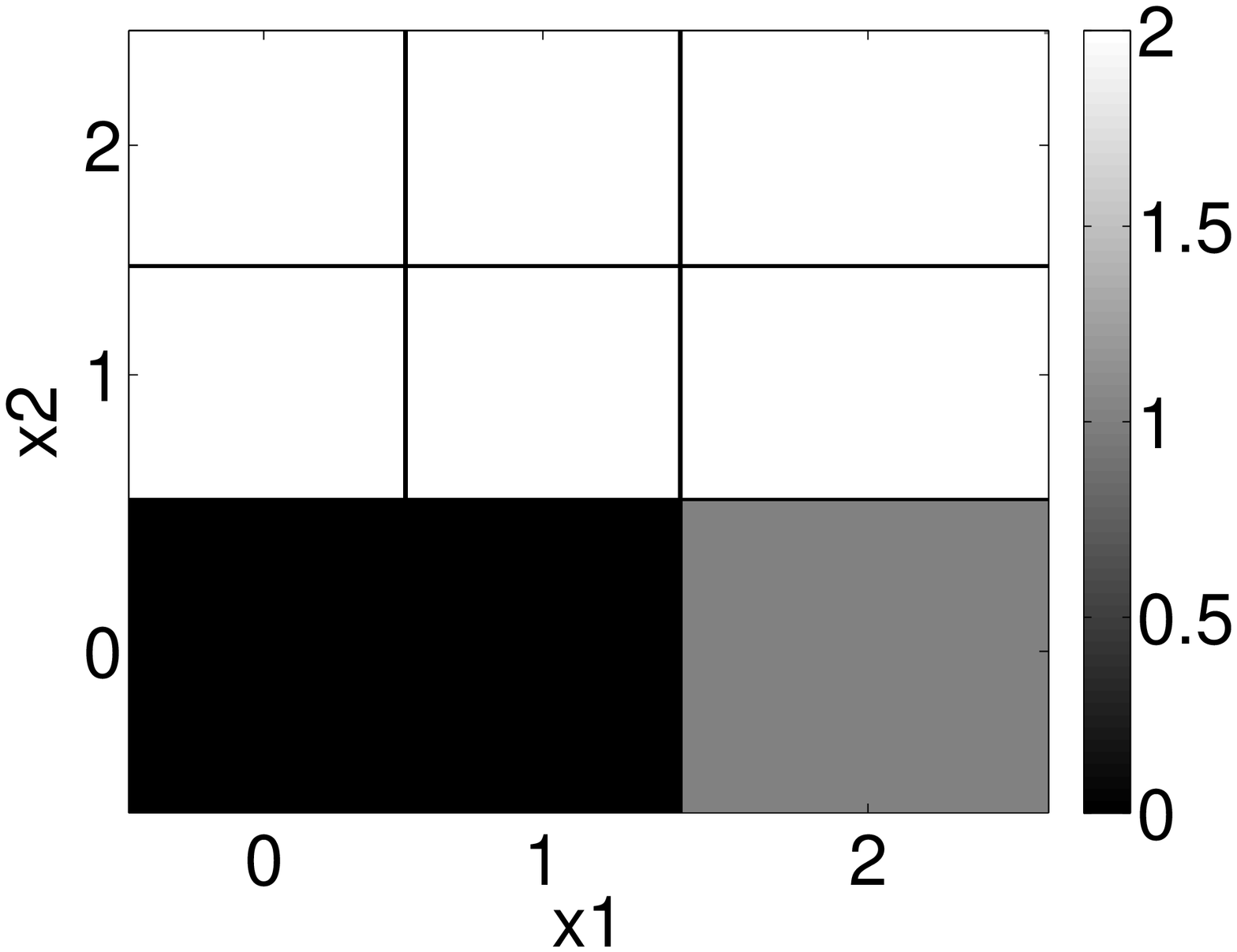}}
}
\caption{Heat maps for function $\f_2$ (left) and $f_2$ (right) from Example \ref{eg:toy}.}
\label{fig:H2}
\end{figure}

Notice how the heat maps have the same qualitative features. The heat maps also show how to divide the domain of $\f$ to capture its qualitative behaviour. For example, the values of $\x_1$ can be separated in three regions: [0,.3[ (low), ].3,.6[ (medium) and ].6,1] (high). Similarly, the values of $\x_2$ can be separated in three regions: [0,.4[ (low), ].4,.7[ (medium) and ].7,1] (high).

\begin{definition}
Consider $k=(k_{1},k_{2},\ldots,k_{m})\in (0,1)^m$ such that $0<k_{1}<k_{2}<\ldots<k_{m}<1$. For each element of $\lb 0,m\rb =\{0,1,\ldots,m\}$ we define $\pmb{[}0\pmb{]}_k=[0,k_1[$, $\pmb{[}1\pmb{]}_k=]k_1,k_2[,\ldots,\pmb{[}m-1\pmb{]}_k=]k_{m-1},k_{m}[$, $\pmb{[}m\pmb{]}_k=]k_m,1]$. The subscript will be omitted if $k$ is understood from the context. The numbers $k_i$'s will be called \textit{thresholds}. 

Now consider $k^i\in (0,1)^{m_i}$ for $i=1,2,\ldots,N$. For each $x\in\prod_{i=1}^n\lb 0,m_i\rb$ we define $\pmb{[}x\pmb{]}=\pmb{[}x_1\pmb{]}_{k^1}\times \pmb{[}x_2\pmb{]}_{k^2}\times\ldots\times \pmb{[}x_N\pmb{]}_{k^N}$. The sets $\pmb{[}x\pmb{]}$'s will be called \textit{regions}. 
\end{definition}

For instance, from Example \ref{eg:toy} we got the thresholds $k^1=(.3,.6)$ for the first variable and $k^2=(.4,.7)$ for the second variable. Then, $\pmb{[}00\pmb{]}=[0,.3[\times[0,.4[$, $\pmb{[}01\pmb{]}=[0,.3[\times].4,.7[$, $\pmb{[}02\pmb{]}=[0,.3[\times].7,1]$ and so on. All regions are shown in Figure \ref{fig:regions}.

\begin{figure}[here]
\centerline{
\framebox{ \includegraphics[height=4cm]{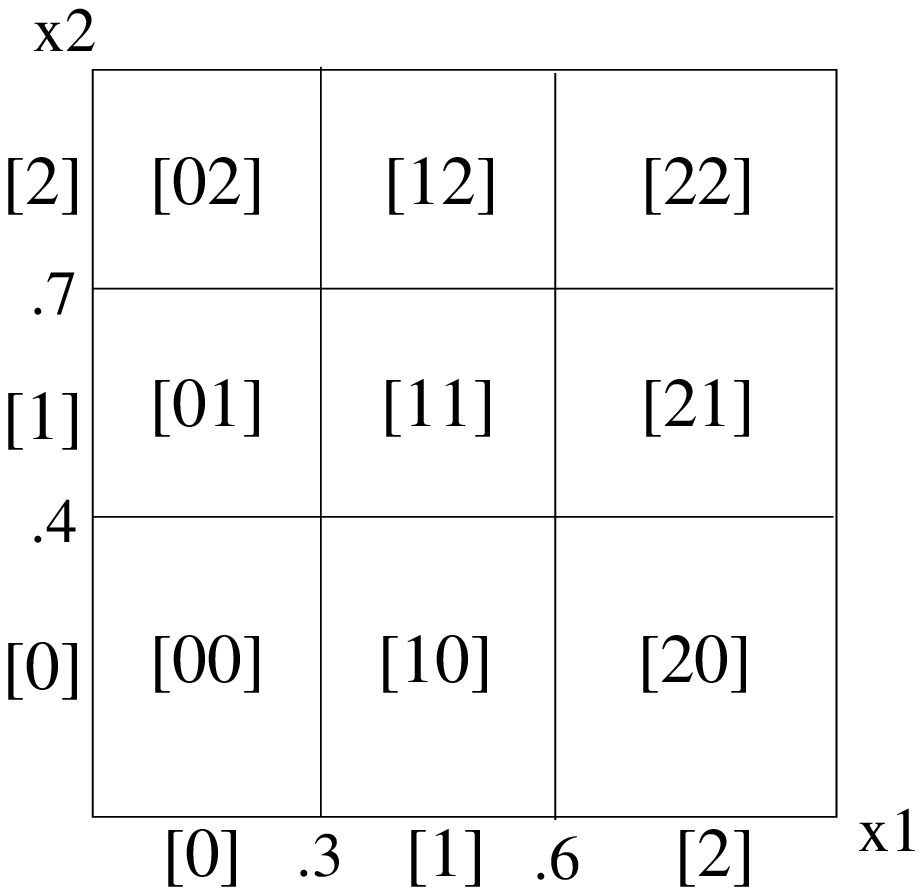}}
}
\caption{Regions for the domain of $\f$ in Example \ref{eg:toy}.}
\label{fig:regions}
\end{figure}

\section{Results}
\label{sec:res}

In this section we will prove that under certain conditions, there is a one-to-one correspondence between the steady states of ODEs and MNs.

We first need the following lemmas.

\begin{lemma}\label{lemma:existence}
Let $\f:[0,1]^N\rightarrow [0,1]^N$ be a continuous function. If $K$ is a convex compact set such that $\f(K)\subseteq K$, then the ODE $\x'=D(\f(\x)-\x)$ has a steady state in $K$.
\end{lemma}
\begin{proof}
It is a direct application of Brouwer's fixed point theorem.
\end{proof}

\begin{lemma}\label{lemma:uniqueness}
Let $\f:[0,1]^N\rightarrow [0,1]^N$ be a differentiable function. If $K$ is a convex subset of $[0,1]^N$, and $\|\f'(\x) \| \leq \epsilon<1$ for all $\x\in K$, then the ODE $\x'=D(\f(\x)-\x)$ has at most one steady state in $K$.
\end{lemma}
\begin{proof}
If $\x,\y\in K$ are steady states of $\x'=D(\f(\x)-\x)$, then $\f(\x)=\x$ and $\f(\y)=\y$. Also, $\|\x-\y\|=\|\f(\x)-\f(\y)\|\leq \epsilon \|\x-\y\|$; it follows that $\x=\y$.
\end{proof}

\begin{lemma}\label{lemma:stability}
Let $\f:[0,1]^N\rightarrow [0,1]^N$ be a differentiable function. If $\|\f'\|<\frac{\min\{D_{ii}\}}{\sqrt{N}\|D\|}$ on $K$, then any steady state of $\x'=D(\f(\x)-\x)$ in $K$ is asymptotically stable.
\end{lemma}
\begin{proof}
Suppose that $\x\in K$ is a steady state of $\x'=D(\f(\x)-\x)$ and denote $B=D\f'(\x)$. Notice that the Jacobian matrix of $D(\f-\x)$ is $B-D$. Suppose $\lambda$ is an eigenvalue of $B-D$. Then, by the Gershgorin circle theorem, there is $k$ such that $|\lambda-(B_{kk}-D_{kk})|\leq \sum_{j\neq k}|B_{kj}-D_{kj}|$. Since $D$ is a diagonal matrix, we obtain $|\lambda-(B_{kk}-D_{kk})|\leq \sum_{j\neq k}|B_{kj}|$. It follows that $|\lambda+D_{kk}|\leq \sum|B_{kj}|\leq \sqrt{N}\|B\|$. Then, $|\lambda+D_{kk}|$ $\leq \sqrt{N}\|B\|=\sqrt{N}\|\f'(\x)D\|$ $\leq \sqrt{N}\|\f'(\x)\|\|D\|$ $<\min\{D_{ii}\}$ $\leq D_{kk}$; thus we obtain $|\lambda+D_{kk}|<D_{kk}$. It follows that $\lambda$ must have negative real part. Therefore, $\x$ is an asymptotically stable steady state.
\end{proof}

Using the lemmas above we can easily prove the following theorem that relates steady states of continuous and discrete networks.   Condition (1) states that the continuous and discrete functions have the same qualitative behaviour (for $n$ large enough). Condition (2) states that the continuous function is ``sigmoidal enough''. We use $n$ to denote a vector parameter and limits refer to all entries of $n$ going to $\infty$; for instance, in Example \ref{eg:toy} $n=(n_1,\ldots,n_6)$ and limits refer to $n_i\rightarrow\infty$ for $i=1,2,\ldots,6$. The parameter $n$ controls how sigmoidal the function is.

\begin{theorem}\label{thm:main}
Let $(\f^n)$ be a family of continuous functions from $[0,1]^N$ to itself and let $f:\mathcal{S}\rightarrow \mathcal{S}$ be a MN. Consider the following conditions:

\begin{itemize}
\item[(1)] $\f^n\xrightarrow{unif} L_x\in \pmb{[}f(x)\pmb{]}$ on compact subsets of $\pmb{[}x\pmb{]}$, for all $x\in\mathcal{S}$.
\item[(2)] ${\f}^{n'} \xrightarrow{unif} 0$ on compact subsets of $\pmb{[}x\pmb{]}$, for all $x\in\mathcal{S}$.
\end{itemize}

Now, for each $x\in\mathcal{S}$ consider a convex compact set $K_x\subseteq \pmb{[}x\pmb{]}$ such that $L_x$ is an interior point of $K_{f(x)}$ (with the topology inherited from $[0,1]^N$) and let $K=\bigcup_{x\in\mathcal{S}}K_x$. 

If \emph{(1)} holds, then for $n$ large enough there is a one-to-one correspondence between steady states in $K$ of the ODE $\x'=D(\f^n(\x)-\x)$ and  $f$. Furthermore, there is a steady state in $K_x$ if and only if $x$ is a steady state of $f$. Also, if $\x^n\in K_x$ is the steady state of $\f^n$, we have $\x^n\rightarrow L_x$.
Additionally, if \emph{(2)} holds, such steady states are unique and asymptotically stable. 

\end{theorem}
	
\begin{proof}
Suppose (1) holds. 
Since $\f^n\xrightarrow{unif} L_x$ on $K_x$ and $L_x$ is an interior point of $K_{f(x)}$ for all $x\in\mathcal{S}$, there exists $n_0$ such that for all $n\geq n_0$ and for all $x\in\mathcal{S}$, $\f^n(K_x)\subseteq K_{f(x)}$. If $x$ is a steady state of $f$, that is $f(x)=x$, then $\f^n(K_x)\subseteq K_x$; by Lemma \ref{lemma:existence} there exists a steady state of the ODE in $K_x$. On the other hand, if $x$ is not a steady state of $f$, then $K_x\cap \f^n(K_x)\subseteq K_x\cap K_{f(x)}\subseteq \pmb{[}x\pmb{]}\cap\pmb{[}f(x)\pmb{]}=\{\}$ and there cannot be a steady state in $K_x$.

Now, for $x\in \mathcal{S}$ such that $f(x)=x$, denote with $\x^n$ a steady state of $\f^n$ in $K_x$. Then, since $\f^n\xrightarrow{unif} L_x$ on $K_x$, we have that $\x^n=\f^n(\x^n)\rightarrow L_x$.

Now suppose that (2) also holds.
Consider $\epsilon<\frac{\min\{D_{ii}\}}{\sqrt{N}\|D\|}\leq 1$ Since ${\f}^{n'} \xrightarrow{unif} 0$ on $K_x$, there exists $n_1\geq n_0$ such that for all $n\geq n_1$, $\|\f^{n'}\|<\epsilon$ on $K$. Then, by Lemma \ref{lemma:uniqueness} and \ref{lemma:stability}, the steady state must be unique and asymptotically stable.
\end{proof}

\begin{example}\label{eg:toy2} It is straightforward to check that the function in Example \ref{eg:toy} satisfies the following ($K$ denotes a compact set):

$$\f\xrightarrow[\textrm{on }K\subseteq {\pmb{[}}x{\pmb{]}} ]{unif} L_x=\begin{cases}
( \ 0,\ 0) \in \pmb{[}00\pmb{]}=\pmb{[}f(00)\pmb{]}    , & \pmb{[}x\pmb{]}= \pmb{[}00\pmb{]} \\
(\ 0,.9)   \in \pmb{[}02\pmb{]}=\pmb{[}f(01)\pmb{]}    , & \pmb{[}x\pmb{]}= \pmb{[}01\pmb{]} \\
(.6,.9)    \in \pmb{[}12\pmb{]}=\pmb{[}f(02)\pmb{]}    , & \pmb{[}x\pmb{]}= \pmb{[}02\pmb{]} \\
(.8,\ 0)   \in \pmb{[}20\pmb{]}=\pmb{[}f(10)\pmb{]}    , & \pmb{[}x\pmb{]}= \pmb{[}10\pmb{]} \\
(.8,.9)    \in \pmb{[}22\pmb{]}=\pmb{[}f(11)\pmb{]}    , & \pmb{[}x\pmb{]}= \pmb{[}11\pmb{]} \\
(.8,.9)    \in \pmb{[}22\pmb{]}=\pmb{[}f(12)\pmb{]}    , & \pmb{[}x\pmb{]}= \pmb{[}12\pmb{]} \\
(.8,.5)    \in \pmb{[}21\pmb{]}=\pmb{[}f(20)\pmb{]}    , & \pmb{[}x\pmb{]}= \pmb{[}20\pmb{]} \\
(.8,.9)    \in \pmb{[}22\pmb{]}=\pmb{[}f(21)\pmb{]}    , & \pmb{[}x\pmb{]}= \pmb{[}21\pmb{]} \\
(.8,.9)    \in \pmb{[}22\pmb{]}=\pmb{[}f(22)\pmb{]}    , & \pmb{[}x\pmb{]}= \pmb{[}22\pmb{]} 
\end{cases} 
$$
Also, it is easy to see that $\f'\xrightarrow{unif} 0$ (as $n\rightarrow\infty$) on compact subsets of any region. Therefore, Theorem \ref{thm:main} guarantees that for $n$ large enough we have a one to one correspondence between steady states and they are asymptotically stable. Since 00 and 22 are the steady states of $f$, the ODE has two stable steady states. For $n=(2,2,2,2,2,2)$ we obtain the steady states $(0,0)\in \pmb{[}00 \pmb{]}$ and $(.73,.78)\in \pmb{[}22\pmb{]}$, shown in Figure \ref{fig:dyn}. The phase portrait of $f$ is constructed by placing an arrow from $x\in\mathcal{S}$ to $y\in\mathcal{S}$ if $f(x)=y$; steady states are denoted by dots.
\end{example}

\begin{figure}[here]
\centerline{
\framebox{\includegraphics[height=5cm]{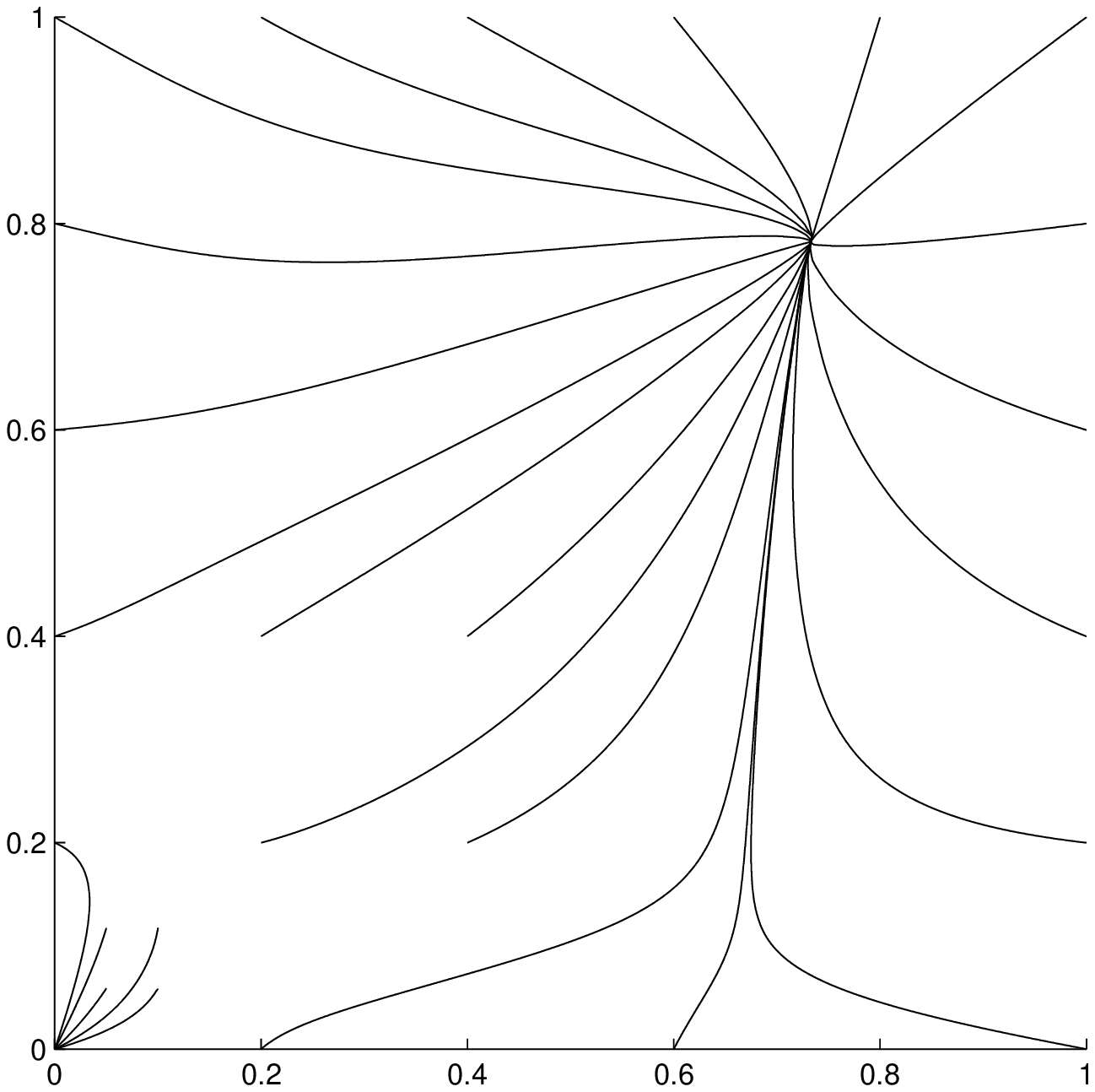} \ \includegraphics[height=5cm]{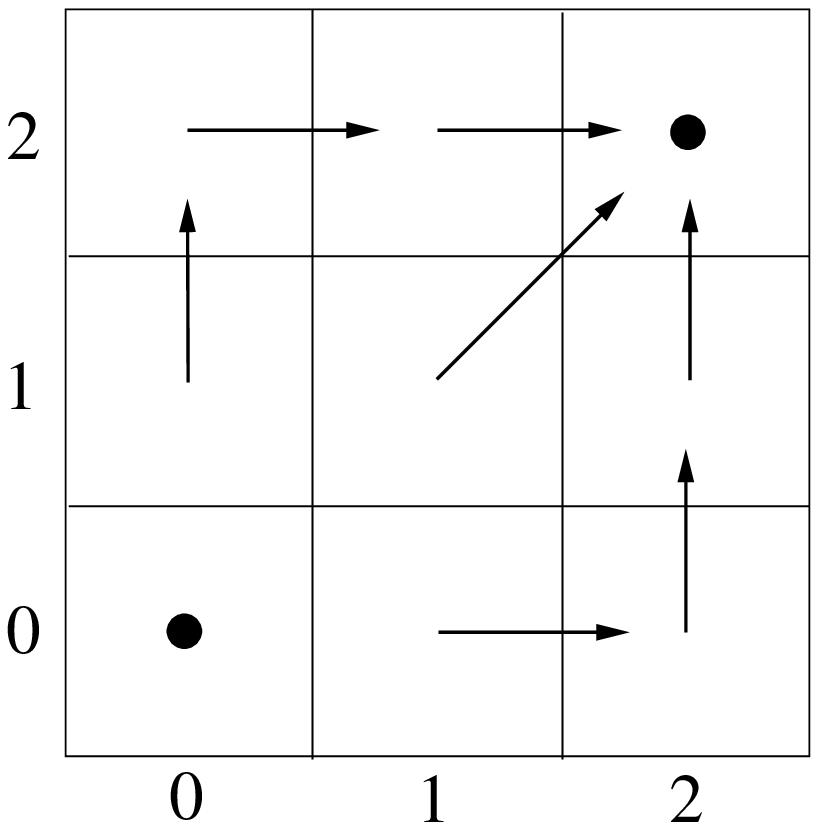}}
}
\caption{Phase portrait of $\f$ (left) and $f$ (right). Steady states of $f$ are denoted by dots.}
\label{fig:dyn}
\end{figure}

It is important to mention that Theorem \ref{thm:main} generalizes previous results. For example, by restricting the theorem to piecewise-linear differential equations we obtain Theorem 1 in \cite{plinear}; by restricting the theorem to Boolean networks and Hill functions we obtain Theorem 2 in \cite{BtoC}. 

The following corollary states that the set where we might not have a one-to-one correspondence can be made as small as possible.

\begin{corollary}\label{cor:01minusK}
Suppose that condition (1) from Theorem \ref{thm:main} is satisfied and consider $\epsilon>0$. 
There exists $K\subseteq [0,1]^N$ with $\mu([0,1]^N\setminus K)\leq\epsilon$ ($\mu=$Lebesgue measure) such that for $n$ large enough there is a one-to-one correspondence between steady states in $K$ of the continuous and discrete network.
\end{corollary}

Another implication of Theorem \ref{thm:main} is that the steady states of an ODE $\x'=D(\f(x)-\x)$ are located either near $L_x$ or near the thresholds. This approach was used in \cite{mendoza} to estimate the stable steady states of an ODE model for Th-cell differentiation using as a starting point the steady states of a discrete model. Our results support this heuristic approach.

It is important to mention that although Theorem \ref{thm:main} indicates that $n$ has to be large, which would be meaningless for real parameters in biological regulation, the conditions for Lemma \ref{lemma:existence}, \ref{lemma:uniqueness}, \ref{lemma:stability} can be satisfied in practice for low values of $n$. For instance, it turns out that the conclusion of Theorem \ref{thm:main} holds for values of $n_i$ as low as $2$ (see Figure \ref{fig:dyn}).

In other words, a continuous function can be sigmoidal enough for biological meaningful parameters. This can explain why many continuous and discrete models of biological systems have similar behaviour \cite{p53sode,p53lm,vondassow,albert, santillan2008, Velizlacop}. This also supports the conjecture that the dynamics of biological systems strongly depend on the ``logic'' of the regulation and not on the exact kinetic parameters, \cite{thomasbook,albert}; in particular, the dynamical behaviour of biological systems is very robust to changes in the parameters.

\section{Application}
\label{sec:app}

In this section we show how our results can be used to gain understanding on how the dynamics of a continuous model depends on the wiring diagram. Other applications of results relating discrete and continuous models have been shown in \cite{BtoC,mendoza}.

It is a well known fact that the topology of a network can constrain its dynamics. For example, it has been shown that positive feedback loops are responsible for multistationarity \cite{soule,jacobianbn,Richard20093281}.

We now will use our results to apply a theorem about MN to ODEs. First we need the following terminology.
Consider a (signed directed) graph, $G=(V_G,E_G)$ with vertices $V_G$ and edges $E_G$. A \textit{positive feedback vertex set} (PFVS) is a set $P\subset V_G$ such that it intersects all positive feedback loops. In \cite{Aracena_FB_BRN} and \cite{Richard20093281} the authors proved the following theorem for Boolean networks and MN, respectively.

\begin{theorem}\label{thm:PFVS} Let $f:\mathcal{S}=\prod_{j=1}^N \lb 0,m_j\rb \rightarrow\mathcal{S}$ be a MN and let $P$ be a PFVS. Then, the number of steady states is bounded by $\prod_{i\in P}(m_i+1)$. Notice that in the Boolean case we have the bound $2^{|P|}$.
\end{theorem}

Combining our results and Theorem \ref{thm:PFVS} we easily obtain the following result.

\begin{theorem}
Under the assumptions of Theorem \ref{thm:main} and for $n$ large enough, the number of steady states in $K$ of $\x'=D(\f^n(\x)-\x)$ is at most $\prod_{i\in P}(m_i+1)$; where $P$ is a PFVS of the wiring diagram and the set $[0,1]^N\setminus K$ can be made as small as required.
\end{theorem}

\begin{example}
Consider the differential equation $\x'=D(\f(\x)-\x)$ where $D$ is a positive diagonal matrix and $f$ is given by:\\
$
\begin{array}{llllllll}
\f_1 = \frac{.5^{n_1}}{.5^{n_1}+\x_4^{n_1}} & ,& \f_2 = \frac{.5^{n_2}}{.5^{n_2}+\x_1^{n_2}} \frac{\x_3^{n_3}}{.5^{n_3}+\x_3^{n_3}} & ,& \f_3 = \frac{.5^{n_4}}{.5^{n_4}+\x_6^{n_4}}\\
\f_4 = \frac{.5^{n_5}}{.5^{n_5}+\x_5^{n_5}} & , & \f_5 = \frac{.5^{n_6}}{.5^{n_6}+\x_2^{n_6}} \frac{\x_7^{n_7}}{.5^{n_7}+\x_7^{n_7}} \frac{\x_8^{n_8}}{.5^{n_8}+\x_8^{n_8}}& ,
 & \f_6 = \frac{\x_5^{n_9}}{.5^{n_9}+\x_5^{n_9}}\\
\f_7 = \frac{.5^{n_{10}}}{.5^{n_{10}}+\x_4^{n_{10}}}\frac{\x_5^{n_{11}}}{.5^{n_{11}}+\x_5^{n_{11}}} & , & \f_8 = \frac{\x_7^{n_{12}}}{.5^{n_{12}}+\x_7^{n_{12}}}\frac{.5^{n_{12}}}{.5^{n_{12}}+\x_9^{n_{12}}}& , 
& \f_9 = \frac{.5^{n_{13}}}{.5^{n_{13}}+\x_6^{n_{13}}}
\end{array}
$
\end{example}

\begin{figure}[here]
\centerline{
\framebox{ \includegraphics[height=3cm]{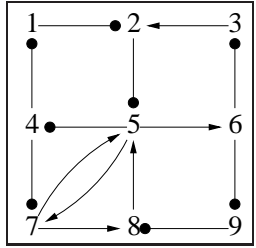}}
}
\caption{Wiring diagram of $\f$ in Example \ref{eg:toy2}.}
\label{fig:wd_PFVS}
\end{figure}

The wiring diagram is shown in Figure \ref{fig:wd_PFVS}. Similarly to Example \ref{eg:toy2}, it is not difficult to check that $\f$ has the same qualitative properties as a Boolean function (the thresholds are $.5$ in this case). Also, it is easy to check that the set $P=\{5\}$ is a PFVS. Therefore, by Theorem \ref{thm:PFVS}, for $n$ large enough we have at most $2^{|P|}=2$ stable steady states in $K$.

\section{Discussion}

The problem of relating continuous and discrete models has been studied by several authors, \cite{kauffman-glass73,BtoC,mendoza,plinear}. Previous results focused on piecewise linear and Boolean functions. We have shown that for sigmoidal ODEs, a steady state in the discrete model gives rise to a steady state in the continuous model; furthermore, this steady state is unique and asymptotically stable. Our results generalize previous results \cite{kauffman-glass73,BtoC,mendoza,plinear}.

One application of our results is the ability to extend the applicability of tools about discrete models to continuous models as shown in Section \ref{sec:app}. The problem of relating network topology to dynamics has been studied extensively for discrete models; those results can give new insight on how network topology constrains the dynamics in continuous models. 

A natural question arising from our work is whether or not one can obtain similar results about periodic solutions. For some biological systems and in very particular cases it has been shown that continuous and discrete models produce similar periodic behaviour \cite{glass75,glass10,p53sode,p53lm,plinear}. General results in this direction would increase our understanding of the relationship of continuous and discrete models, how the dynamical properties are constrained by topological features of the wiring diagram and how biology works at the system level. This deserves further investigation.

\end{document}